\UseRawInputEncoding
\documentclass[11pt,a4paper]{article}
\usepackage{pgf}
\usepackage{tikz}
\usepackage[title]{appendix}
\usetikzlibrary{arrows,automata}
\usepackage{verbatim}
\usepackage{caption}
\usepackage{epsf,epsfig,amsfonts,amsgen,amsmath,amstext,amsbsy,amsopn,amsthm,amssymb}
\usepackage{color}
\usepackage{mathrsfs,hyperref}
\usepackage{tikz}
\usepackage{enumerate}
\usepackage{url}
\usepackage{enumitem}
\usetikzlibrary{patterns}

\usepackage{setspace}

\newtheorem{thm}{Theorem}[section]

\newtheorem{lem}[thm]{Lemma}

\def\QED{\hfill \rule{7pt}{7pt}}

\newcommand{\dR}{{\mathcal{R}}}

\def\ex{\mathrm{ex}}
\makeatletter
\newcommand{\rmnum}[1]{\romannumeral #1}
\newcommand{\Rmnum}[1]{\expandafter\@slowromancap\romannumeral #1@}
\makeatother
\begin{document}
	\title{The minimum number of clique-saturating edges}
	
	\author{Jialin He$^1$\and
		Fuhong Ma$^1$\and
		Jie Ma$^1$\and
		Xinyang Ye$^2$}	
	\date{}
	
	\maketitle
	
	\footnotetext[1]{School of Mathematical Sciences, University of Science and Technology of China, Hefei, Anhui 230026, China. Partially supported by the National Key R and D Program of China 2020YFA0713100, National Natural Science Foundation of China grant 12125106, and Anhui Initiative in Quantum Information Technologies grant AHY150200.}

	\footnotetext[2]{School of Mathematical Sciences, Peking University, Beijing 100871, China.}

	\begin{abstract}
		Let $G$ be a $K_p$-free graph.
		We say $e$ is a	$K_p$-saturating edge of $G$ if $e\notin E(G)$ and $G+e$ contains a copy of $K_p$.
		Denote by $f_p(n, e)$ the minimum number of $K_p$-saturating edges that an $n$-vertex $K_p$-free graph with $e$ edges can have.
		Erd\H{o}s and Tuza conjectured that
		$f_4(n,\lfloor n^2/4\rfloor+1)=\left(1 + o(1)\right)\frac{n^2}{16}.$
		Balogh and Liu disproved this by showing
		$f_4(n,\lfloor n^2/4\rfloor+1)=(1+o(1))\frac{2n^2}{33}$.
		They believed that a natural generalization of their construction for $K_p$-free graph should also be optimal and made a conjecture that
		$f_{p+1}(n,\ex(n,K_p)+1)=\left(\frac{2(p-2)^2}{p(4p^2-11p+8)}+o(1)\right)n^2$ for all integers $p\ge 3$.
		The main result of this paper is to confirm the above conjecture of Balogh and Liu.

	\end{abstract}

\section{Introduction}
Given a graph $H$, we say a graph $G$ is {\it $H$-free} if $G$ does not contain $H$ as a subgraph.
Let the {\it Tur\'an number} $\ex(n,H)$ of $H$ denote the maximum number of edges in an $n$-vertex $H$-free graph.
The study of Tur\'an numbers can date back to the work of Mantel \cite{M07} and is the central subject in extremal graph theory (see \cite{FS-survey} for a recent survey).
The classical theorem of Tur\'an \cite{T41} states that for any integer $p\ge 2$, the unique $n$-vertex $K_{p+1}$-free graph attaining the maximum number $\ex(n,K_{p+1})$ of edges is the $p$-partite {\it Tur\'an graph} $T_p(n)$, i.e., the $n$-vertex complete balanced graph.


For $p\geq 3$, let $G$ be a $K_{p}$-free graph and $e$ be a {\it non-edge} of $G$ (i.e., an edge in the complement of $G$).
We say $e$ is a {\it $K_{p}$-saturating edge} of $G$, if $G+e$ contains a copy of $K_p$.
This notion is closely related to Tur\'an numbers.
Indeed, a $K_p$-free graph $G$ is maximal if and only if
every non-edge of $G$ is a $K_{p}$-saturating edge (let us call this property $\star$).
So in other words, Tur\'an's Theorem determines the maximum number of edges $e(G)$ over all $K_p$-free graphs $G$ satisfying the property $\star$.
On the other hand, Zykov \cite{Z49} and independently Erd\H os, Hajnal and Moon \cite{EHM} determined the minimum number $e(G)$ over all $n$-vertex $K_p$-free graphs $G$ satisfying the property $\star$,
which is uniquely attained by the $n$-vertex complement graph of a clique of size $n-p+2$.
For more references on this minimization problem, we refer interested readers to the recent surveys \cite{CFFS,G19} and to \cite{BBM,BBMO} for related problems in the language of graph bootstrap percolation.

In this paper, we consider another type of extremal problems on the clique-saturating edges.
For a $K_p$-free graph $G$,
let $f_p(G)$ denote the number of $K_p$-saturating edges of $G$.
Let $f_p(n, e)$ be the minimum number of $K_p$-saturating edges of an $n$-vertex $K_p$-free graph with $e$ edges.
For all integers $p\geq 3$, the example of the Tur\'an graph $T_{p-1}(n)$ shows that
$$f_{p+1}(n,e)=0 \mbox{ ~ for all ~} 0\le e\le \ex(n, K_p).$$
Erd\H{o}s and Tuza (see \cite{E90}) proved that $f_4(n,\lfloor \frac{n^2}4 \rfloor+1)\ge cn^2$ for some constant $c>0$; that is, for the case $p=3$, if adding one more edge to the above extreme, then the function will suddenly jump from $0$ to $\Omega(n^2)$.
Erd\H{o}s and Tuza also made a conjecture that
$f_4\left(n,\left\lfloor \frac{n^2}4\right\rfloor+1\right)=\left(1 + o(1)\right)\frac{n^2}{16}.$
This however was disproved by Balogh and Liu in \cite{BL},
where they constructed an $n$-vertex $K_4$-free graph with $\lfloor\frac{n^2}4\rfloor+1$ edges and with only $(1+o(1))\frac{2n^2}{33}$ $K_4$-saturating edges (see Figure~1 in the case $p=3$ for the construction).
Furthermore, Balogh and Liu \cite{BL} showed that this construction is best possible.

\begin{thm}[Balogh-Liu \cite{BL}]\label{Thm:B-L}
	$f_4(n,\lfloor \frac{n^2}4 \rfloor+1)=\left(1 + o(1)\right)\frac{2n^2}{33}.$
\end{thm}

In fact, they proved a stronger statement that $f_4(n,\lfloor \frac{n^2}4 \rfloor+t)=\frac{2}{33}n^2+\Theta(n)$ for every $1\le t\le \frac{n}{66}.$
Balogh and Liu~\cite{BL} commented that a similar phenomenon like Theorem~\ref{Thm:B-L} should also hold for general $p$ and thus made an explicit conjecture (see Remark~(iii) in \cite{BL}) suggested by a natural generalization of their $K_4$-free construction that for all integers $p\ge 3,$
$$f_{p+1}\big(n,\ex(n, K_p)+1\big)=\left(\frac{2(p-2)^2}{p(4p^2-11p+8)}+o(1)\right)n^2.$$

The main result of the present paper is to prove the above conjecture of Balogh and Liu~\cite{BL}.

\begin{thm}\label{Thm:main}
	For all integers $p\ge 3,$
	$f_{p+1}\big(n,\ex(n, K_p)+1\big)=\left(\frac{2(p-2)^2}{p(4p^2-11p+8)}+o(1)\right)n^2.$
\end{thm}

Most of this paper will be devoted to the lower bound of the following theorem, which implies the lower bound of Theorem~\ref{Thm:main}. Note that for any integer $p\geq 3$, $f_{p+1}(G)=0$ holds for $G=T_{p-1}(n)$.

\begin{thm}\label{Thm:thorough version}
	Let $p\ge 3$ and $n\ge 8 p^5$ be integers.
	Let $\mathcal{G}$ be the family consisting of all $n$-vertex $K_{p+1}$-free graphs with exactly $\ex(n,K_p)$ edges.
	Then
	$$\min_{G\in \mathcal{G}\backslash\{T_{p-1}(n)\}} f_{p+1}(G)=\frac{2(p-2)^2}{p(4p^2-11p+8)}n^2-\frac{(p-2)(2p-3)}{4p^2-11p+8}n+O_p(1).$$
	In addition, if $n$ is divisible by $p(p-1)(4p^2-11p+8),$ then
	$$\min_{G\in \mathcal{G}\backslash\{T_{p-1}(n)\}} f_{p+1}(G)=\frac{2(p-2)^2}{p(4p^2-11p+8)}n^2-\frac{(p-2)(2p-3)}{4p^2-11p+8}n.$$
\end{thm}

We refer readers to the beginning of Section~4 for a proof sketch of this theorem.

We use standard notations on graphs throughout the paper.
Let $G$ be a graph.
For a subset $U\subseteq V(G)$, the subgraph of $G$ induced by the vertex set $U$ is denoted by $G[U]$,
while the subgraph obtained from $G$ by deleting all vertices in $U$ is expressed by $G\backslash U$.
Let $N_{G}(U)=\bigcap_{v\in U} N_{G}(v)$ be {\it the common neighborhood} of all vertices of $U$ in $G$.
Suppose that $U,W$ are two disjoint vertex subsets in $G$.
We denote $E_{G}(U,W)$ to be the set of edges of $G$ between $U$ and $W$ and let $e_{G}(U,W)=|E_{G}(U,W)|$.
We often drop the above subscripts when they are clear from context.
For positive integers $k$, we write $[k]$ for the set $\{1,2,...,k\}$ and
the notation $\binom{x}{2}$ means the function $x(x-1)/2$ for all reals $x$.
We often omit floors and ceilings whenever they are not critical.

The rest of the paper is organized as follows.
In Section~2, we provide constructions which match with the upper bounds of Theorems~\ref{Thm:main} and \ref{Thm:thorough version}.
In Section~3, we prove Theorem~\ref{Thm:main} by using Theorem~\ref{Thm:thorough version}.
In Section~4, we give a complete proof of Theorem~\ref{Thm:thorough version}.
Finally, we conclude the paper with some remarks.

	\section{The constructions for the upper bounds}
	
	In this section, we establish the upper bounds of Theorems~\ref{Thm:main} and \ref{Thm:thorough version} by defining some explicit $K_{p+1}$-free graphs.
	These graphs are suggested by Balogh and Liu in \cite{BL},
	each of which is an appropriate blow-up of the following graph:
	take a complete $(p-1)$-partite graph $K=K_{2,...,2}$ and add a new vertex by making it adjacent to exactly one vertex in each partite set of $K$.

	In the rest of this section, we write $n=p(p-1)(4p^2-11p+8)x+y,$ where $x, y$ are integers such that $x\geq 0$ and $0\le y<p(p-1)(4p^2-11p+8).$
	By Tur\'an's Theorem, we have $\ex(n,K_p)=\frac{p-2}{2(p-1)}\cdot p^2(p-1)^2(4p^2-11p+8)^2x^2+p(p-2)(4p^2-11p+8)xy+t_{p-1}(y),$ where $t_{p-1}(y)=e(T_{p-1}(y)).$

	First, we prove the desired upper bound of Theorem~\ref{Thm:thorough version}.
	
	\medskip	
	
	\noindent\textbf{The upper bound of Theorem~\ref{Thm:thorough version}.}
	In this case, we will construct an $n$-vertex $K_{p+1}$-free graph $H_1$ with exactly $\ex(n,K_p)$ edges and $f_{p+1}(H_1)=\frac{2(p-2)^2}{p(4p^2-11p+8)}n^2-\frac{(p-2)(2p-3)}{p(4p^2-11p+8)}n+O_p(1)$.
	
	To do so, we first construct a graph $H_0$ as follows (see Figure 1).
	First, take a $(p-1)$-partite complete graph $K_{2,...,2}$ with vertex set $\{v_i,u_i:i\in[p-1]\},$ where $v_i$ and $u_i$ are in the same part for $i\in[p-1]$.
	Next, take a new vertex $v_0$ and make it adjacent to each $v_i$ for $i\in [p-1]$.
	Finally, let $H_0$ be obtained by blowing-up $v_0$ into an independent set $V_0$ of size $2(p-1)(p-2)^2x,$
	blowing-up each $v_i$ into an independent set $V_i$ of size $4(p-1)^2(p-2)x$ for $i\in [p-1],$
	blowing-up each $u_i$ into an independent set $U_i$ of size $p(3p-4)x$ for $i\in[p-1].$
	We can check that $H_0$ is $K_{p+1}$-free on $p(p-1)(4p^2-11p+8)x$ vertices with $\frac{p-2}{2(p-1)}\cdot p^2(p-1)^2(4p^2-11p+8)^2x^2$ edge.\footnote{All the detailed calculations in this proof can be found in Appendix A.}
	
	Next we construct the desired graph $H_1$ from $H_0$ by enlarging the size of $V_0$ with $2y$ more new vertices and deleting $y$ vertices which form a $T_{p-1}(y)$ in $H_0\left[\bigcup_{i=1}^{p-1} U_i\right].$
	Indeed, since $n\ge 8p^5$ and by our definitions of $x$ and $y,$
	we can check that $p(p-1)(3p-4)x>y.$
	Thus, the above deletion process succeed.
	By a careful calculation, we can derive that
	$H_1$ is $K_{p+1}$-free on $n$ vertices with $\ex(n,K_p)$ edges.
	The only $K_{p+1}$-saturating edges are the pairs in $V_i$ for $0\le i\le p-1$. This (see Appendix A) leads to
    \begin{align*}	
	f_{p+1}(H_1)&=\frac{2(p-2)^2}{p(4p^2-11p+8)}n^2-\frac{(p-2)(2p-3)}{p(4p^2-11p+8)}n+\frac{8(p-1)^3}{p(4p^2-11p+8)}y^2-\frac{2(p-1)^2}{4p^2-11p+8}y\\
	&=\frac{2(p-2)^2}{p(4p^2-11p+8)}n^2-\frac{(p-2)(2p-3)}{4p^2-11p+8}n+O_p(1),
    \end{align*}
	completing the proof for the upper bound. \QED

	\begin{center}
		\begin{tikzpicture}[-,>=stealth',shorten >=1pt,auto,node distance=50pt, semithick,scale=1.3]
			
			\node[circle, scale=1.15pt, draw=black, line width=1pt, outer sep=0]                 (A) at (0,1.7)         {$V_0$};
			\node[circle, scale=1.4pt, draw=black, line width=1pt, outer sep=0]                   (B) at (-2.5,0)      {$V_1$};
			\node[circle, scale=1.4pt, draw=black, line width=1pt, outer sep=0]                   (C) at (-1.1,0)      {$V_2$};
			\node[circle, scale=0.3pt, fill=black]                                              (D) at (-0.3,0)      {};
			\node[circle, scale=0.3pt, fill=black]                                              (E) at (0,0)         {};
			\node[circle, scale=0.3pt, fill=black]                                              (F) at (0.3,0)       {};
			\node[circle, scale=1.4pt, draw=black, line width=1pt, inner sep=0.3pt, outer sep=0]  (G) at (1.1,0)       {$V_{p-2}$};
			\node[circle, scale=1.4pt, draw=black, line width=1pt, inner sep=0.3pt, outer sep=0]  (H) at (2.5,0)       {$V_{p-1}$};
			\node[circle, scale=0.8pt, draw=black, line width=1pt, outer sep=0]                   (I) at (-2.5,-2)     {$U_1$};
			\node[circle, scale=0.8pt, draw=black, line width=1pt, outer sep=0]                   (J) at (-1.1,-2)     {$U_2$};
			\node[circle, scale=0.3pt, fill=black]                                              (K) at (-0.3,-2)     {};
			\node[circle, scale=0.3pt, fill=black]                                              (L) at (0,-2)        {};
			\node[circle, scale=0.3pt, fill=black]                                              (M) at (0.3,-2)      {};
			\node[circle, scale=0.8pt, draw=black, line width=1pt, inner sep=0.3pt, outer sep=0]  (N) at (1.1,-2)      {$U_{p-2}$};
			\node[circle, scale=0.8pt, draw=black, line width=1pt, inner sep=0.3pt, outer sep=0]  (O) at (2.5,-2)      {$U_{p-1}$};
			
			\draw (0.1,2.05) node [above,scale=1.2pt] {$2(p-1)(p-2)^2x$};
			\draw (-3,0) node [left,scale=1.2pt] {$4(p-1)^2(p-2)x$};
			\draw (3,0) node [right,scale=1.2pt] {$4(p-1)^2(p-2)x$};
			\draw (2.8,-2) node [right,scale=1.1pt] {$p(3p-4)x$};
			\draw (-2.8,-2) node [left,scale=1.1pt] {$p(3p-4)x$};

			\path (A) edge[line width=2.2pt] (B);
			\path (A) edge[line width=2.2pt] (C);
			\path (A) edge[line width=2.2pt] (G);
			\path (A) edge[line width=2.2pt] (H);
			
			\path (B) edge[line width=2.2pt] (C);
			\path (B) edge[bend left, line width=2.2pt] (H);
			\path (B) edge[bend left, line width=2.2pt] (G);
			\path (C) edge[bend left, line width=2.2pt] (H);
			\path (C) edge[bend left, line width=2.2pt] (G);
			\path (G) edge[line width=2.2pt] (H);

			\path (B) edge[line width=2.2pt] (J);
			\path (B) edge[line width=2.2pt] (N);
			\path (B) edge[line width=2.2pt] (O);
			\path (C) edge[line width=2.2pt] (I);
			\path (C) edge[line width=2.2pt] (N);
			\path (C) edge[line width=2.2pt] (O);
			\path (G) edge[line width=2.2pt] (I);
			\path (G) edge[line width=2.2pt] (J);
			\path (G) edge[line width=2.2pt] (O);
			\path (H) edge[line width=2.2pt] (I);
			\path (H) edge[line width=2.2pt] (J);
			\path (H) edge[line width=2.2pt] (N);
			
			\path (I) edge[line width=2.2pt] (J);
			\path (I) edge[bend right, line width=2.2pt] (N);
			\path (I) edge[bend right, line width=2.2pt] (O);
			\path (J) edge[bend right, line width=2.2pt] (N);
			\path (J) edge[bend right, line width=2.2pt] (O);
			\path (N) edge[line width=2.2pt] (O);

		\end{tikzpicture}
		
	\end{center}
	
	\begin{center}
		\textbf{Figure 1.} Constructions for the upper bounds of Theorem \ref{Thm:main} and \ref{Thm:thorough version}.
	\end{center}

	The construction for the upper bound of Theorem~\ref{Thm:main} is quite similar to the one above. The only differences are the sizes of the parts in the blow-up.
	
	\bigskip
	
	\noindent\textbf{The upper bound of Theorem~\ref{Thm:main}.}
	In this case, we will construct an $n$-vertex $K_{p+1}$-free graph $G$ with exactly $\ex(n,K_p)+1$ edges and $f_{p+1}(G)=\frac{2(p-2)^2}{p(4p^2-11p+8)}n^2-\frac{(p-2)(2p^2-5p+4)}{p(4p^2-11p+8)}n+O_p(1).$
	
	To do so, we first let $H_0$ be the same as above (see Figure 1).
	We then construct a graph $H_2$ from $H_0$ by enlarging the size of $V_0$ with $2y+1$ more new vertices and deleting $y+1$ vertices which forms a $T_{p-1}(y+1)$ in $H_0\left[\bigcup_{i=1}^{p-1} U_i\right].$
	By a similar calculation as the previous case, one can derive that
	$H_2$ is $K_{p+1}$-free on $n$ vertices with $$\ex(n,K_p)+\frac{(p-2)^3}{p(p-1)(4p^2-11p+8)}n+O_p(1)$$ edges.
	Again, the only $K_{p+1}$-saturating edges are the pairs in $V_i$ for $0\le i\le p-1$.
	By some careful calculation, one can derive that
	$$f_{p+1}(H_2)=\frac{2(p-2)^2}{p(4p^2-11p+8)}n^2-\frac{(p-2)(2p^2-5p+4)}{p(4p^2-11p+8)}n+O_p(1).$$
	Now one can easily remove some edges from $H_2$ (i.e., edges incident with vertices in $U_i$'s)
	without changing the number of $K_{p+1}$-saturating edges until the remaining graph $G$ has exactly $\ex(n,K_p)+1$ edges.
	In this way, we obtain the desired graph $G$ with $f_{p+1}(G)=f_{p+1}(H_2)$ and thus prove that $$f_{p+1}\big(n,\ex(n, K_p)+1\big)\leq\frac{2(p-2)^2}{p(4p^2-11p+8)}n^2-\frac{(p-2)(2p^2-5p+4)}{p(4p^2-11p+8)}n+O_p(1)$$
	holds for all integers $n$.\QED

	\section{Proof of Theorem \ref{Thm:main}}
	In this section, assuming Theorem~\ref{Thm:thorough version}, we complete the proof of Theorem \ref{Thm:main}.
	The upper bound of Theorem \ref{Thm:main} is given by the last section, so it suffices to prove the lower bound.
	Let $G$ be a $K_{p+1}$-free graph with $\ex(n,K_p)+1$ edges. By Tur\'an's Theorem, $G$ contains a copy of $K_p.$
	Let $G'$ be obtained from $G$ by removing a single edge such that
	$G'$ still contains a $K_p.$
	Then $G'$ is $K_{p+1}$-free with $\ex(n,K_p)$ edges.
	As $G'$ contains a $K_p,$ it cannot be the Tur\'an graph $T_{p-1}(n)$.
	By Theorem \ref{Thm:thorough version}, we have
	$$f_{p+1}(G)\ge f_{p+1}(G')\ge\frac{2(p-2)^2}{p(4p^2-11p+8)}n^2-\frac{(p-2)(2p-3)}{4p^2-11p+8}n+O_p(1),$$
	finishing the proof of Theorem~\ref{Thm:main}. \QED
	\medskip
	
	We remark that the above proofs actually show a sharper bound than the statement of Theorem~\ref{Thm:main}. Namely, for all integers $p\ge 3$ and $n$, if we write
	$$g_p(n)=f_{p+1}\big(n,\ex(n,K_p)+1\big)-\frac{2(p-2)^2}{p(4p^2-11p+8)}n^2,$$
	then we have $$-\frac{(p-2)(2p-3)}{4p^2-11p+8}n+O_p(1)\le g_p(n)\le -\frac{(p-2)(2p^2-5p+4)}{p(4p^2-11p+8)}n+O_p(1)$$ such that $g_p(n)\rightarrow -\frac{n}{2}+o(n)$ as $p\rightarrow\infty.$

\section{Proof of Theorem \ref{Thm:thorough version}}
We begin with a sketch of the proof of Theorem \ref{Thm:thorough version}.
Let $G$ be an $n$-vertex $K_{p+1}$-free graph with $\ex(n,K_p)$ edges and containing at least one copy of $K_p$.
Following the approach of \cite{BL}, we partition the vertex set of $G$ into two parts $V(\dR)$ and its complement $V(G)\backslash V(\dR)$, where $\dR$ is a maximum family of vertex-disjoint $K_p$'s in $G$ and $V(\dR)$ denotes the set of all vertices contained in $\dR$.
Then all $K_p$-saturating edges of $G$ can be divided into two types, the first type of which are those $K_p$-saturating edges incident to $V(\dR)$ and the other type are those contained in $V(G)\backslash V(\dR).$
Estimations on the number of saturating edges of these two types have been established in \cite{BL}, respectively, which work quite well when $p$ is small.
The problem is that when $p$ is getting bigger, the complexity of computations based on these estimations will be difficult to handle.
So some novel ideas will be needed.
A key motivation for us comes after Lemma~\ref{Lem:<p-joint-book},
which roughly says that for any $p$-clique $R$ in $\dR$,
as long as there are enough edges between $R$ and $V(G)\backslash V(\dR)$,
any $p-1$ vertices of $R$ have some common neighbors in $V(G)\backslash V(\dR)$ (it can even be set up as $\Omega(1)$ many if required).
Therefore one may hope to use similar proof ideas as in Hajnal-Szemer\'edi Theorem to find a larger collection of vertex-disjoint $p$-cliques than $\dR$ and thus obtain a contradiction to the maximality of $\dR$.
This indeed would work.
And it turns out that if we choose $\dR$ with an additional requirement on the number of edges contained in $V(G)\backslash V(\dR)$,
then the proof can be shortened and a contradiction can already be reached using some appropriate vertex-switching techniques (in fact this would provide a shortcut for the proof in \cite{BL} for the case $p=3$ as well).

Throughout the rest of this section, we present the proof of Theorem \ref{Thm:thorough version}.
The upper bounds of Theorem \ref{Thm:thorough version} are given by the aforementioned constructions.
Consider any integers $p\ge 3$ and $n\ge 120p^2.$ (As $n\ge 8p^5\ge 120 p^2,$ here we remark that $n\ge 120 p^2$ is enough to show the lower bounds of Theorem \ref{Thm:thorough version}).
Let $G$ be any $n$-vertex $K_{p+1}$-free graph with $\ex(n,K_p)$ edges, but not the $(p-1)$-partite Tur\'an graph $T_{p-1}(n)$.
It suffices to show that $f_{p+1}(G)$ is bounded from below by the desired formula.
By Tur\'an's Theorem, $G$ contains at least one copy of $K_p$ with
\begin{equation}\label{equ:e(G)}
	e(G)=\ex(n, K_p)=\frac{p-2}{2(p-1)}n^2-\delta,
\end{equation}
where $\delta=\frac{t(p-1-t)}{2(p-1)}$ for $t\in \{0,1,...,p-2\}$ with $t\equiv n\mod (p-1)$.
We note that $0\leq \delta\leq \frac{p-1}{8}$, and
$\delta=0$ if and only if $n$ is divisible by $p-1$.

	We now partition $V(G)$ into two parts $V(\dR)$ and $V(G)\backslash V(\dR)$ satisfying the following conditions
\begin{itemize}
	\item [(\rmnum{1}).] $\dR$ is a maximum family of vertex-disjoint $K_p$'s in $G$, and\footnote{Throughout we will write $V(\dR)$ for the union of the vertex sets of all $K_p$'s in $\dR.$}
	\item [(\rmnum{2}).] subject to (\rmnum{1}), the remaining graph $G\backslash V(\dR)$ has the maximum number of edges.
\end{itemize}
Let $H_\dR:=G\backslash V(\dR)$ and $|\dR|:=rn$. Since $G$ contains a $K_p$, we have
\begin{equation}\label{Equ:range of r}
	1/n\le r\le 1/p.
\end{equation}
By the choice of (\rmnum{1}), we know that $H_\dR$ is $K_p$-free with $(1-pr)n$ vertices, thus by Tur\'an's Theorem,
\begin{equation}\label{Equ:e(H)}
	e(H_\dR)\le \frac{(p-2)}{2(p-1)}(1-pr)^2n^2.
\end{equation}

	For any $p$-clique $R\in \dR$ and $0\le j\le p,$
	we let $$Z_j(R)=\{\mbox{all vertices in $H_{\dR}$ that has exactly $j$ neighbors in $V(R)$}\} \mbox{ ~and ~} z_j(R):=|Z_j(R)|/n.$$
	By the assumption that $G$ is $K_{p+1}$-free, it is clear that $Z_p(R)=\emptyset$.
	So for any $p$-clique $R\in \dR$,
	\begin{equation}\label{Equ:sum z_j}
		\sum_{j=0}^{p-1}z_j(R)=1-pr.
	\end{equation}
    We will also need to consider a refined partition of $Z_{p-1}(R)$ as follows.
    Let $\{v_1,v_2,...,v_{p}\}$ represent the vertex set of a given $p$-clique $R\in \dR$.
    For any $i\in [p],$ define $$A_i(R):= N_{H_{\dR}}(R\backslash\{v_i\})$$ to be the common neighborhood of $V(R)\backslash \{v_i\}$ in $V(H_{\dR})$.
    Let us observe that $A_i(R)$'s are pairwise vertex-disjoint independent sets in $Z_{p-1}(R)$ (for otherwise
    $\big(\bigcup_i A_i(R)\big)\cup R$ would contain a copy of $K_{p+1}$, a contradiction to $G$ is $K_{p+1}$-free).
    In particular, we have
    \begin{equation}\label{Equ:sum A_i}
    	\sum_{i=1}^{p}|A_i(R)|/n=z_{p-1}(R).
    \end{equation}
    It is crucial to see that every non-edge inside each $A_i(R)$ is a $K_{p+1}$-saturating edge in $G$.

    	\begin{center}
    	\begin{tikzpicture}[-,>=stealth',shorten >=1pt,auto,node distance=50pt, semithick,scale=1.4]
    		
    		\draw [line width=2pt, pattern=north west lines] (0.3,-1) rectangle (1,1);
    		\draw [line width=2pt] (-1,-1) rectangle (0.3,1);
    		\draw [line width=2pt, pattern=north east lines] (2.5,-1) rectangle (3.2,1);
    		\draw [line width=2pt] (3.2,-1) rectangle (5.8,1);		
    		\draw [color=red, -{>[sep=0pt]}, line width=2pt] (0.66,0.6) to [bend left=20] (2.8,0.6);
    		\draw [color=red, -{>[sep=0pt]}, line width=2pt] (2.8,-0.6) to [bend left=20] (0.66,-0.6);
    		\draw (0.05,1.5) node [below ,scale=1.2pt] {$R$};
    		\draw (4.25,1.5) node [below ,scale=1.2pt] {$H_{\dR}$};
    		\draw (0.35,-1) node [below right,scale=1.2pt] {$C$};
    		\draw (0.25,-1) node [below left,scale=1.2pt] {$R\backslash C$};
    		\draw (2.6,-1) node [below right,scale=1.2pt] {$C'$};
    		\draw (5,-1) node [below left,scale=1.2pt] {$H_{\dR}\backslash C'$};
    		
    	\end{tikzpicture}
    \end{center}

    \begin{center}
    	\textbf{Figure 2.} The proof of Lemma~\ref{Lem:Key lemma for p-joint-book}
    \end{center}

	The following lemma is key in our proof. It shows that by the choice of $\dR$ and $H_{\dR}$,
there are enough many edges incident to new $p$-cliques obtained from some $R\in \dR$ by switching some vertices in $R$ with vertices in $H_{\dR}$ of equal size.

\begin{lem}\label{Lem:Key lemma for p-joint-book}
	Let $R\in \dR$ be a $p$-clique and $C$ be a subclique of $R$.
	If there exists a clique $C'$ in $H_{\dR}$ of equal size as $C$ such that $R':=(R\backslash C)\cup C'$ remains a clique in $G,$
	then $\dR':=\left(\dR\backslash\{R\}\right)\cup\{R'\}$ is also a maximum family of vertex-disjoint $K_p$'s in $G$ with
	$e(R',H_{\dR'})\ge e(R,H_{\dR}),$ where $H_{\dR'}=G\backslash V(\dR')$
\end{lem}

	\begin{proof}
		First observe that $\dR'$ is also a maximum family of $rn$ vertex-disjoint $K_p$'s.
		Let $H_{\dR'}=G\backslash V(\dR')$. So $H_{\dR'}=(H_\dR\backslash C')\cup C$ (see Figure 2).
		By (ii), we have $e(H_{\dR})\ge e(H_{\dR'}).$
		Since $e(C')=e(C)$,
		\begin{center}
			$e(H_{\dR})=e(C')+e(C', H_{\dR}\backslash C')+e(H_{\dR}\backslash C')$ and  $e(H_{\dR'})=e(C)+e(C, H_{\dR}\backslash C')+e(H_{\dR}\backslash C')$,
		\end{center}
	it follows that
	\begin{equation}
		e\left(C',H_{\dR}\backslash C'\right)\ge e\left(C,H_{\dR}\backslash C'\right).\nonumber
	\end{equation}
Therefore, as $e(R\backslash C,C')=e(R\backslash C,C)$, one can derive that
\begin{align*}
	e(R',H_{\dR'})-e(R,H_{\dR})
	=e\left(C',H_{\dR}\backslash C' \right)- e\left(C,H_{\dR}\backslash C' \right)
	\ge 0.
\end{align*}
This completes the proof of Lemma \ref{Lem:Key lemma for p-joint-book}.
	\end{proof}

	Next we proceed to prove three technical lemmas and we should emphasize in advance that these lemmas hold for any family $\dR$ solely satisfying the condition~(\rmnum{1}). The first one says that for any family $\dR$ satisfying the condition~(\rmnum{1}), there is a $R^*\in \dR$ such that $e(R^*,H_{\dR})$ is large.

\begin{lem}\label{Lem:e(R^*,H)}
	Suppose that $\dR$ is under the condition (\rmnum{1}) and $H_{\dR}=G\backslash V(\dR).$
	Then there exists a $p$-clique $R^*\in \dR$ such that
	\begin{equation}\label{Equ:e(R^*,H)}
		e(R^*,H_{\dR})\ge \left(\frac{p(p-2)}{p-1}-\frac{p(2p^2-4p+1)}{2(p-1)}r\right)n-\frac{\delta}{rn}.
	\end{equation}
	Moreover, for any $R^*\in \dR$ satisfying \eqref{Equ:e(R^*,H)}, it holds that
	\begin{equation}\label{Equ:z{p-1}}
		z_{p-1}(R^*)
		\ge\frac{p-2}{p-1}-\frac{p(2p-3)}{2(p-1)}r-\frac{\delta}{rn^2}.
	\end{equation}
\end{lem}
	
	\begin{proof}
	Note that the edge set of $G$ can be partitioned into $E(H_{\dR}),$ $E(V(\dR),H_{\dR})$ and $E(G[V(\dR)]).$
	Since $G$ is $K_{p+1}$-free, by Tur\'an's Theorem $e(G[V(\dR)])\le \ex(prn,K_{p+1})= \binom p 2 r^2n^2.$
	Together with \eqref{equ:e(G)} and \eqref{Equ:e(H)}, we have that
	\begin{align*}
		e(V(\dR),H_{\dR})
		&=e(G)-e(H_{\dR})-e(G[V(\dR)])\\
		&\ge
		\left(\frac{p-2}{2(p-1)}-\frac{(p-2)}{2(p-1)}(1-pr)^2-\frac{p(p-1)}{2}r^2\right)n^2-\delta\\
		&=\left(\frac{p(p-2)}{p-1}r-\frac{p(2p^2-4p+1)}{2(p-1)}r^2\right)n^2-\delta.
	\end{align*}
	By averaging, there exists a clique $R^*\in\dR$ with $$e(R^*,H_{\dR})\ge \frac{e(V(\dR),H_{\dR})}{rn}
	\ge\left(\frac{p(p-2)}{p-1}-\frac{p(2p^2-4p+1)}{2(p-1)}r\right)n-\frac{\delta}{rn}.$$
	As $G$ is $K_{p+1}$-free, every vertex in $H_{\dR}$ has at most $r-1$ neighbors in $V(R^*)$. So we have $e(R^*,H_\dR)\leq |Z_{p-1}(R^*)|+(p-2)\sum_{j=0}^{p-1}|Z_j(R^*)|$,
	which by \eqref{Equ:sum z_j} implies that
	$$z_{p-1}(R^*)\ge \frac{e(R^*,H_{\dR})}{n}-(p-2)(1-pr) \ge \frac{p-2}{p-1}-\frac{p(2p-3)}{2(p-1)}r-\frac{\delta}{rn^2}.$$
	This completes the proof of Lemma \ref{Lem:e(R^*,H)}.
\end{proof}

	Denote by $\ell_1^{\dR}$ the number of $K_{p+1}$-saturating edges incident to $V(\dR)$,
and by $\ell_2^{\dR}$ the number of $K_{p+1}$-saturating edges in $H_{\dR}$.
Obviously $f_{p+1}(G)=\ell_1^{\dR}+\ell_2^{\dR}$.
The lemma below gives a lower bound on $\ell_1^{\dR}$, which in particular shows that Theorem~\ref{Thm:thorough version} holds in case $r$ is close to $1/p$.
\begin{lem}\label{Lem:r_1n^2}
	Suppose that $\dR$ is under the condition (\rmnum{1}).
	Then
	$$\ell_1^{\dR}
	\ge\left(\frac{p-2}{p-1}r-\frac{p(p-2)}{2(p-1)}r^2\right)n^2-\frac{pr}{2}n-\delta.$$
	Moreover, if $r>\frac{2(p-2)(2p-3)}{p(4p^2-11p+8)},$ then Theorem~\ref{Thm:thorough version} holds.
\end{lem}
	\begin{proof}
	Let $\dR=\{R_1, R_2,...,R_{rn}\}$, $H_{\dR}=G\backslash V(\dR)$ and $r_i=e(R_i,G\backslash\bigcup_{j=1}^iR_j)$ for $i\in [rn]$ such that
	$$\sum_{i=1}^{rn}r_i=e(G)-e(H_{\dR})-\binom p 2 rn.$$
	Since $G$ is $K_{p+1}$-free, every vertex has at most $p-1$ neighbors on each $R_i.$
	So there exist at least $r_i-(p-2)(n-pi)$ vertices in $G\backslash\bigcup_{j=1}^iR_j$ with exactly $p-1$ neighbors in $V(R_i),$ each of which contributes a $K_{p+1}$-saturating edges to $\ell_1^{\dR}$.
	Therefore, we have
	\begin{align*}
		\ell_1^{\dR}
		&\ge \sum_{i=1}^{rn}\left(r_i-(p-2)(n-pi)\right)\\
		&=\left(e(G)-e(H_{\dR})-\binom p 2 rn\right)
		-\left((p-2)rn^2-\frac{p(p-2)}{2}(rn+1)rn\right)\\
		&\ge\left(\frac{p-2}{2(p-1)}-\frac{(p-2)}{2(p-1)}(1-pr)^2-(p-2)r+\frac{p(p-2)}{2}r^2\right)n^2-\frac{pr}{2}n-\delta\\
		&=\left(\frac{p-2}{p-1}r-\frac{p(p-2)}{2(p-1)}r^2\right)n^2-\frac{pr}{2}n-\delta,
	\end{align*}
	where the last inequality follows from \eqref{equ:e(G)} and \eqref{Equ:e(H)}.
	
	For the second statement of this lemma, by \eqref{Equ:range of r} and the assumption therein, we have $\frac{2(p-2)(2p-3)}{p(4p^2-11p+8)}< r\le \frac 1p.$
	Then the first statement implies that
	\begin{align*}
		f_{p+1}(G)&\ge \ell_1^{\dR}
		\ge\left(\frac{p-2}{p-1}r-\frac{p(p-2)}{2(p-1)}r^2\right)n^2-\frac{pr}{2}n-\delta\\
		&\ge \frac{p-2}{p-1}\cdot r\left(1-\frac{pr}{2}\right)n^2-\frac{n}{2}-\delta\\
		&> \left(\frac{2(p-2)^2(2p-3)}{p(p-1)(4p^2-11p+8)}-\frac{2(p-2)^3(2p-3)^2}{p(p-1)(4p^2-11p+8)^2}\right)n^2-\frac{n}{2}-\delta\\
		&= \frac{4(p-1)(p-2)^2(2p-3)}{p(4p^2-11p+8)^2}n^2-\frac{n}{2}-\delta\\
		&\ge \frac{2(p-2)^2}{p(4p^2-11p+8)}n^2-\frac{(p-2)(2p-3)}{4p^2-11p+8}n-\delta,
	\end{align*}
	where the second last inequality holds because $g(r)= r(1-\frac{pr}{2})$ is increasing for $r\leq \frac1p$ and the last inequality holds whenever $n\ge 120p^2$. 
	This matches the lower bounds of Theorem~\ref{Thm:thorough version} (also for the case when $n$ is divisible by $p(p-1)(4p^2-11p+8)$, as for which $\delta=0$).
	Now Lemma \ref{Lem:r_1n^2} is completed.
\end{proof}

	The next lemma says that for any $R^*\in \dR$ satisfying the conclusion of Lemma~\ref{Lem:e(R^*,H)}, one may assume that the set $A_i(R^*)$ for every $i\in [p]$ is non-empty.
\begin{lem}\label{Lem:<p-joint-book}
	Suppose that $\dR$ is under the condition (\rmnum{1}). Let $R^*\in \dR$ be any clique satisfying \eqref{Equ:e(R^*,H)}. If there exists some $i\in [p]$ such that $A_i(R^*)=\emptyset$, then Theorem~\ref{Thm:thorough version} holds.
\end{lem}
	
	\begin{proof}
	Let $A_i=A_i(R^*)$ and $z_{p-1}=z_{p-1}(R^*)$.
	Without loss of generality, we assume that $A_p=\emptyset$.
	Recall that each pair of vertices in $A_i$ is a $K_{p+1}$-saturating edge in $H_{\dR}$ and by \eqref{Equ:sum A_i}, $\sum_{i=1}^{p-1}|A_i|/n=\sum_{i=1}^{p}|A_i|/n=z_{p-1}$.
	Using Jensen's inequality, we get that
	\begin{align*}
		\ell_2^{\dR}\ge \sum_{i=1}^{p-1}\binom{|A_i|}{2}\ge(p-1)\binom{\frac{z_{p-1}}{p-1}n}{2}=\frac{z_{p-1}^2}{2(p-1)}n^2-\frac{z_{p-1}}{2}n.
	\end{align*}
	Since $\dR$ satisfies the condition (\rmnum{1}), by Lemma~\ref{Lem:r_1n^2},
	we may assume that $r\le \frac{2(p-2)(2p-3)}{p(4p^2-11p+8)}$.
	By \eqref{Equ:range of r}, we have $r\geq 1/n$.
	Since $\delta\le \frac{p-1}{8}$, we can derive from \eqref{Equ:z{p-1}} that
	\begin{eqnarray*}
		z_{p-1} &\ge& \frac{p-2}{p-1}-\frac{p(2p-3)}{2(p-1)}r-\frac{\delta}{rn^2}    \nonumber     \\
		&\ge& \frac{p-2}{p-1}-\frac{p(2p-3)}{2(p-1)}\cdot\frac{2(p-2)(2p-3)}{p(4p^2-11p+8)}-\frac{p-1}{8n}\nonumber  \\
		&=&\frac{p-2}{4p^2-11p+8}-\frac{p-1}{8n} > \frac{p-1}{2n},  \nonumber
	\end{eqnarray*} where the last inequality holds as $n\ge 120p^2.$
	Note that $h(z_{p-1})=\frac{z_{p-1}^2}{2(p-1)}n^2-\frac{z_{p-1}}{2}n$ is increasing in the range of $z_{p-1}>\frac{p-1}{2n}$ and takes its minimum at the smallest value that $z_{p-1}$ can take.
	Thus
	\begin{eqnarray}\label{Equ:lower bound of l_2}
		\ell_2^{\dR}&\ge&h(z_{p-1})\geq \frac{\left(\frac{p-2}{p-1}-\frac{p(2p-3)}{2(p-1)}r-\frac{\delta}{rn^2}\right)^2}{2(p-1)}n^2-\frac{\left(\frac{p-2}{p-1}-\frac{p(2p-3)}{2(p-1)}r-\frac{\delta}{rn^2}\right)}{2}n   \nonumber\\
		&=&\frac{\left(2(p-2)-p(2p-3)r\right)^2}{8(p-1)^3}n^2-\frac{2(p-2)-p(2p-3)r}{4(p-1)}n+\delta\cdot F(n,p,r,\delta),
	\end{eqnarray}
	where
	\begin{eqnarray}\label{Equ:F(n,p,r,delta)}
		F(n,p,r,\delta)=\frac{\delta}{2(p-1)r^2n^2}-\frac{p-2}{(p-1)^2r}+\frac{p(2p-3)}{2(p-1)^2}+\frac{1}{2rn}\ge -\frac{p-2}{(p-1)^2r}.
	\end{eqnarray}
	Thus, using $r\geq 1/n$ and $\delta\le \frac{p-1}{8}$, we have
	\begin{eqnarray*}
		\ell_2^{\dR}&\ge&
		\frac{\left(2(p-2)-p(2p-3)r\right)^2}{8(p-1)^3}n^2-\frac{p-2}{2(p-1)}n-\frac{(p-2)\delta}{(p-1)^2r}\nonumber\\
		&\geq&
		\frac{\left(2(p-2)-p(2p-3)r\right)^2}{8(p-1)^3}n^2-\frac{p-2}{2(p-1)}n-\frac{p-2}{8(p-1)}n\nonumber\\
		&>&
		\frac{\left(2(p-2)-p(2p-3)r\right)^2}{8(p-1)^3}n^2-\frac{5}{8}n\nonumber.
	\end{eqnarray*}
	Next we claim that Theorem \ref{Thm:thorough version} holds in case $r\le \frac{1}{40p(p-2)(2p-3)}$. Indeed, by the above lower bound of $\ell_2^{\dR}$, we have
	\begin{align*}
		f_{p+1}(G)&\ge \ell_2^{\dR}>
		\frac{\left(2(p-2)-p(2p-3)r\right)^2}{8(p-1)^3}n^2-\frac{5}{8}n\nonumber\ge \frac{\left(2(p-2)-\frac{1}{40(p-2)}\right)^2}{8(p-1)^3}n^2-\frac{5}{8}n\\
		&> \frac{4(p-2)^2-0.1}{8(p-1)^3}n^2-\frac{5}{8}n\ge \frac{2(p-2)^2}{p(4p^2-11p+8)}n^2-\frac{(p-2)(2p-3)}{4p^2-11p+8}n
	\end{align*}
	where the last inequality holds whenever $n\ge 120p^2$ and $p\ge 3$ (see the verification in Appendix~B).
	This matches the lower bounds of Theorem~\ref{Thm:thorough version} and thus proves the above claim.
	
	Therefore in the following of the proof, we may assume that $\frac{1}{40p(p-2)(2p-3)}\le r\le \frac{2(p-2)(2p-3)}{p(4p^2-11p+8)}$.
	By~\eqref{Equ:lower bound of l_2} and \eqref{Equ:F(n,p,r,delta)}, we see that
	$F(n,p,r,\delta)\ge -\frac{p-2}{(p-1)^2r}\ge -\frac{40p(p-2)^2(2p-3)}{(p-1)^2}\ge -40(p-2)(2p-3),$
	and thus
	\begin{eqnarray*}
		\ell_2^{\dR}\ge\frac{\left(2(p-2)-p(2p-3)r\right)^2}{8(p-1)^3}n^2-\frac{2(p-2)-p(2p-3)r}{4(p-1)}n-40\delta(p-2)(2p-3).
	\end{eqnarray*}
	This together with the estimation on $\ell_1^{\dR}$ from Lemma~\ref{Lem:r_1n^2} give that
	\begin{align*}
		f_{p+1}(G)=&\ell_1^{\dR}+\ell_2^{\dR}   \\
		\ge&  \left(\frac{p-2}{p-1}r-\frac{p(p-2)}{2(p-1)}r^2\right)n^2-\frac{pr}{2}n-\delta+ \frac{(2(p-2)-p(2p-3)r)^2}{8(p-1)^3}n^2\\
		&-\frac{2(p-2)-p(2p-3)r}{4(p-1)}n-40\delta(p-2)(2p-3)  \\
		=&\frac{p(4p^2-11p+8)n^2}{8(p-1)^3}r^2-\frac{2(p-2)^2n^2+p(p-1)^2n}{4(p-1)^3}r+\frac{(p-2)^2}{2(p-1)^3}n^2-\frac{p-2}{2(p-1)}n\\
		&-\delta\left(40(p-2)(2p-3)+1\right)\\	
		\ge&\frac{2(p-2)^2}{p(4p^2-11p+8)}n^2-\frac{(p-2)(2p-3)}{4p^2-11p+8}n-\frac{p(p-1)}{8(4p^2-11p+8)}-\delta\left(40(p-2)(2p-3)+1\right)\\
		=& \frac{2(p-2)^2}{p(4p^2-11p+8)}n^2-\frac{(p-2)(2p-3)}{4p^2-11p+8}n+O_p(1),
	\end{align*}
	where the last inequality holds since the quadratic function on $r$ formed by the first two terms
	on one side
	is minimized at $r=\frac{2(p-2)^2}{p(4p^2-11p+8)}+\frac{(p-1)^2}{(4p^2-11p+8)n}$ (for the detailed calculation, see Appendix C).
	This matches the lower bounds of Theorem~\ref{Thm:thorough version}.
	For the case when $n$ is divisible by $p(p-1)(4p^2-11p+8)$,
	we have $\delta=0$ and in this case, the optimal $r$ for the last inequality should be chosen as $r=\frac{2(p-2)^2}{p(4p^2-11p+8)}$ so that $rn$ is an integer\footnote{Note that this value of $r$ corresponds to the exact construction in Section~2}; repeating the above calculation, it would exactly imply that
	$f_{p+1}(G)\ge\frac{2(p-2)^2}{p(4p^2-11p+8)}n^2-\frac{(p-2)(2p-3)}{4p^2-11p+8}n$.
	Now Lemma \ref{Lem:<p-joint-book} is completed.
\end{proof}
	
	\medskip
	
	Finally we are ready to finish the proof of Theorem \ref{Thm:thorough version}.
	By Lemma~\ref{Lem:<p-joint-book}, for any $\dR$ satisfying the condition (\rmnum{1}) and for any $R_0\in \dR$ satisfying \eqref{Equ:e(R^*,H)}, we may assume that $A_i(R_0)\neq \emptyset$ for each $i\in [p]$, i.e.,
	any $p-1$ vertices in $V(R_0)$ have at least one common neighbor in $H_{\dR}=G\backslash V(\dR)$.

	Let $R^*\in \dR$ be the $p$-clique obtained from Lemma~\ref{Lem:e(R^*,H)}.
	So $R^*$ satisfies \eqref{Equ:e(R^*,H)}.
	Let	$C$ be a clique in $H_{\dR}$ of maximum size such that
	$R^*\cup C$ contains a $p$-clique $R'$ in $G$ covering all the vertices of $C$.
	Since $A_i(R^*)\neq \emptyset$ for each $i\in [p]$, such a clique $C$ exists in $H_{\dR}$ (for instance, one can just take one vertex in $A_1(R^*)$).
	Let $V(R^*)=\{v_1,...,v_p\}$ and $V(C)=\{x_1,...,x_c\}$ for some integer $c\ge 1$.
	Without loss of generality we may assume that $$V(R')=\{x_1,...,x_c,v_{c+1},...,v_p\}.$$
	In what follows, we should complete the proof by deriving the final contradiction that $c\ge p$.

Suppose that $c\le p-1$.
In this case, we are always able to find a clique in $H_{\dR}$ of larger size than $C$ and satisfying the above conditions required for $C$.
To see this, let $\dR'=\left(\dR\backslash\{R^*\}\right)\cup\{R'\}$ and $H_{\dR'}=G\backslash V(\dR').$
So $\dR'$ also satisfies the condition (\rmnum{1}) and
$$V(H_{\dR'})=(V(H_{\dR})\backslash \{x_1,\ldots, x_c\})\cup \{v_1,\ldots, v_c\}.$$
Applying Lemma~\ref{Lem:Key lemma for p-joint-book} with the clique $R$ therein being $R^*$,
we know that
$$e(R',H_{\dR'})\ge e(R^*,H_{\dR})\ge \left(\frac{p(p-2)}{p-1}-\frac{p(2p^2-4p+1)}{2(p-1)}r\right)n-\frac{\delta}{rn},$$
where the last inequality holds as $R^*$ satisfies \eqref{Equ:e(R^*,H)}.
That says, $R'\in \dR'$ also satisfies \eqref{Equ:e(R^*,H)}.
As discussed earlier, by Lemma~\ref{Lem:<p-joint-book}, any $p-1$ vertices in $V(R')$ have at least one common neighbor in $H_{\dR'}$.
In particular, there exists a vertex $y\in V(H_{\dR'})$ such that it is not adjacent to $v_p$ but is adjacent to all other vertices of $V(R')$.
Obviously, $y\notin \{v_1,\ldots, v_c\}$, since $v_iv_p\in E(G)$ for each $i\in [c]$.
So it must be the case that $y\in V(H_{\dR})\backslash \{x_1,\ldots, x_c\}.$
Now let $C'=\{x_1,\ldots, x_c, y\} \subseteq V(H_{\dR})$.
Then $C'$ is a clique in $H_{\dR}$ of size larger than $C$ such that $C'\cup\{v_{c+1},\ldots, v_{p-1}\}$ is a $p$-clique contained in $R^*\cup C'$ and covering all vertices of $C'$.
This is a contradiction to our choice of $C$.
Therefore, we must have that $c\ge p$.
However, it is also a contradiction to the fact that $H_{\dR}$ is $K_p$-free,
proving Theorem \ref{Thm:thorough version}.
\QED
	
\section{Concluding remarks}	
In this paper, we determine the order of $f_{p+1}\big(n,\ex(n, K_p)+1\big)$, confirming a conjecture of Balogh and Liu~\cite{BL}.
Balogh and Liu proved a stronger result in \cite{BL} that $f_4(n,\lfloor \frac{n^2}4 \rfloor+t)=\frac{2}{33}n^2+\Theta(n)$ holds for every positive integer $t$ up to $\frac{n}{66}.$
We remark that the upper bound construction of Theorem \ref{Thm:main} as well as the proof of Section 3 also show that
$$f_{p+1}\big(n,\ex(n, K_p)+t\big)= \frac{2(p-2)^2}{p(4p^2-11p+8)}n^2+\Theta(n)$$
holds for any integer $1\le t\le \frac{(p-2)^3}{p(p-1)(4p^2-11p+8)}n$.
It is interesting to determine the function of $f_{p+1}(n, m)$ for every integer $m$ between $\ex(n,K_p)$ and $\ex(n,K_{p+1})$.
We would like to ask if for all $m$, the extremal $K_{p+1}$-free graph attaining this minimum number is always obtained from an appropriate blow-up\footnote{Here, we may allow that some vertices are blowing up to empty sets. For instance, a blow-up of $K_p$ counts.} of the same graph suggested in \cite{BL} (i.e., the graph obtained by taking a complete $(p-1)$-partite graph $K=K_{2,...,2}$ and adding a new vertex by making it adjacent to exactly one vertex in each partite set of $K$) by deleting $O(n)$ edges.

Let $H$ be a given graph. For an $H$-free graph $G$, a non-edge of $G$ is called an {\it $H$-saturating edge}, if $G+e$ contains a copy of $H$.
Let $f_H(G)$ denote the number of $H$-saturating edges of $G$ and let $f_H(n,m)$ denote the minimum of $f_H(G)$ over all $H$-free $n$-vertex graphs $G$ with $m$ edges.
It is natural to consider the same minimization problem $f_H(n,m)$ for general $H$.
The following family of graphs seems to be of particular interest.
A pair of two edges $e, f$ in $H$ is called {\it critical} if $\chi(H-\{e,f\})=\chi(H)-2$.
It is clear that such two edges $e, f$ must be vertex-disjoint.
We say a graph $H$ is {\it double-edge-critical} if it contains a critical pair of two edges $e,f$.
We point out that there are many double-edge-critical graphs, for example,
any join obtained from a $p$-clique for $p\geq 4$ and an arbitrary graph is double-edge-critical.
From the definition, we see that such $H$ is also {\it edge-critical},\footnote{A graph $F$ is edge-critical if there exists an edge $e^*$ such that $\chi(F-e^*)=\chi(F)-1$. A classical result of Erd\H{o}s and Simonovits states that for sufficiently large $n$,
if $F$ is an edge-critical graph, then the unique $n$-vertex $F$-free extremal graph for $\ex(n,F)$ is $T_{\chi(F)-1}(n)$.}
so is each of $H-e$ and $H-f$.
Let $\chi(H)=p+1$.
Then it follows that $f_H(n,m)=0$ for all integers $m\leq \ex(n,K_p)$.  	
We believe that the same phenomenon as Theorem~\ref{Thm:main} holds for any double-edge-critical graph $H$, that is, $f_H\big(n,\ex(n,K_p)+1\big)$ would suddenly jump to $\Omega(n^2)$.
We wonder if $f_H\big(n,\ex(n,K_p)+1\big)$ can be determined for every double-edge-critical graph $H$ with $\chi(H)\geq 4$.
	
\bigskip

\noindent{\bf Acknowledgements.}
The authors would like to thank Jozsef Balogh for suggesting the problem of \cite{BL} to the fourth author in 2020 Summer. The fourth author additionally thanks him for fruitful discussions.

	\bibliographystyle{unsrt}

\appendices

\section{Calculations for the upper bound of Theorem~\ref{Thm:thorough version}}

We first calculate $|V(H_0)|$ and $e(H_0)$ (see Figure 1).
It is easy to see that
$|V(H_0)|=2(p-1)(p-2)^2x+4(p-1)^3(p-2)x+p(p-1)(3p-4)x=p(p-1)(4p^2-11p+8)x,$ and
\begin{align*}
	e(H_0)=&8(p-1)^4(p-2)^3x^2+\frac{p-2}{2(p-1)}\cdot (p-1)^2 \left(4(p-1)^2(p-2)+p(3p-4)\right)^2x^2\\
          =&\frac{p-2}{2(p-1)}\cdot p^2(p-1)^2(4p^2-11p+8)^2x^2.
\end{align*}
By our definition of $H_1,$ we get that $|V(H_1)|=|V(H_0)|+y=p(p-1)(4p^2-11p+8)x+y=n,$ and
\begin{align*}
	e(H_1)=&e(H_0)+2y\cdot 4(p-1)^3(p-2)x-y\cdot 4(p-1)^2(p-2)^2x-y \cdot p(p-2)(3p-4)x+t_{p-1}(y)\\
	=&e(H_0)+p(p-2)(4p^2-11p+8)xy+t_{p-1}(y)\\
	=&\frac{p-2}{2(p-1)}\cdot p^2(p-1)^2(4p^2-11p+8)^2x^2+p(p-2)(4p^2-11p+8)xy+t_{p-1}(y)=\ex(n,K_p).
\end{align*}
Since the only $K_{p+1}$-saturating edges are the pairs in $V_i$ for $0\le i\le p-1,$
we get that
\begin{align*}
	f_{p+1}(H_1)=&\binom{2(p-1)(p-2)^2x+2y}{2}+(p-1)\binom{4(p-1)^2(p-2)x}{2}\\
	=&2p(p-1)^2(p-2)^2(4p^2-11p+8)x^2+4(p-1)(p-2)^2xy\\
	&-p(p-1)(p-2)(2p-3)x+2y^2-y\\
	=&\frac{2(p-2)^2}{p(4p^2-11p+8)}\left(p(p-1)(4p^2-11p+8)x+y\right)^2\\
	&-\frac{(p-2)(2p-3)}{4p^2-11p+8} \left(p(p-1)(4p^2-11p+8)x+y\right)\\
	&+2y^2-\frac{2(p-2)^2}{p(4p^2-11p+8)}y^2+\frac{(p-2)(2p-3)}{4p^2-11p+8}y-y\\
	=&\frac{2(p-2)^2}{p(4p^2-11p+8)}n^2-\frac{(p-2)(2p-3)}{4p^2-11p+8}n+\frac{8(p-1)^3}{p(4p^2-11p+8)}y^2-\frac{2(p-1)^2}{4p^2-11p+8}y,
\end{align*}
as desired. \QED

\section{Verifying an inequality in the proof of Lemma \ref{Lem:<p-joint-book}}
We want to verify that the following inequality appeared as the last inequality in the second last paragraph of the proof of Lemma \ref{Lem:<p-joint-book} (see page 10) holds:
\begin{align*}
	\frac{4(p-2)^2-0.1}{8(p-1)^3}n^2-\frac{5}{8}n\ge \frac{2(p-2)^2}{p(4p^2-11p+8)}n^2-\frac{(p-2)(2p-3)}{4p^2-11p+8}n
\end{align*}
for any $n\ge 120p^2$ and $p\ge 3$.
First observe that the above inequality is equivalent to
$$\left(\frac{4(p-2)^2-0.1}{8(p-1)^3}-\frac{2(p-2)^2}{p(4p^2-11p+8)}\right)n\ge \frac{5}{8}-\frac{(p-2)(2p-3)}{4p^2-11p+8},$$
which can be further simplified to
$$\frac{p(4p^2-16p+15.9)(4p^2-11p+8)-16(p-1)^3(p-2)^2}{8p(p-1)^3(4p^2-11p+8)}n\ge\frac{4p^2+p-8}{8(4p^2-11p+8)}.$$
Let $f(p)=p(4p^2-16p+15.9)(4p^2-11p+8)-16(p-1)^3(p-2)^2=4p^4-32.4p^3+97.1p^2-128.8p+64.$
Then we only need to show that
$nf(p)\ge p(p-1)^3(4p^2+p-8).$
First from the calculation by python (see Figure A), we can see that $f(p)\ge 0$ for all $p\ge 3.$

\begin{figure}
	\includegraphics[height=8cm,width=8cm]{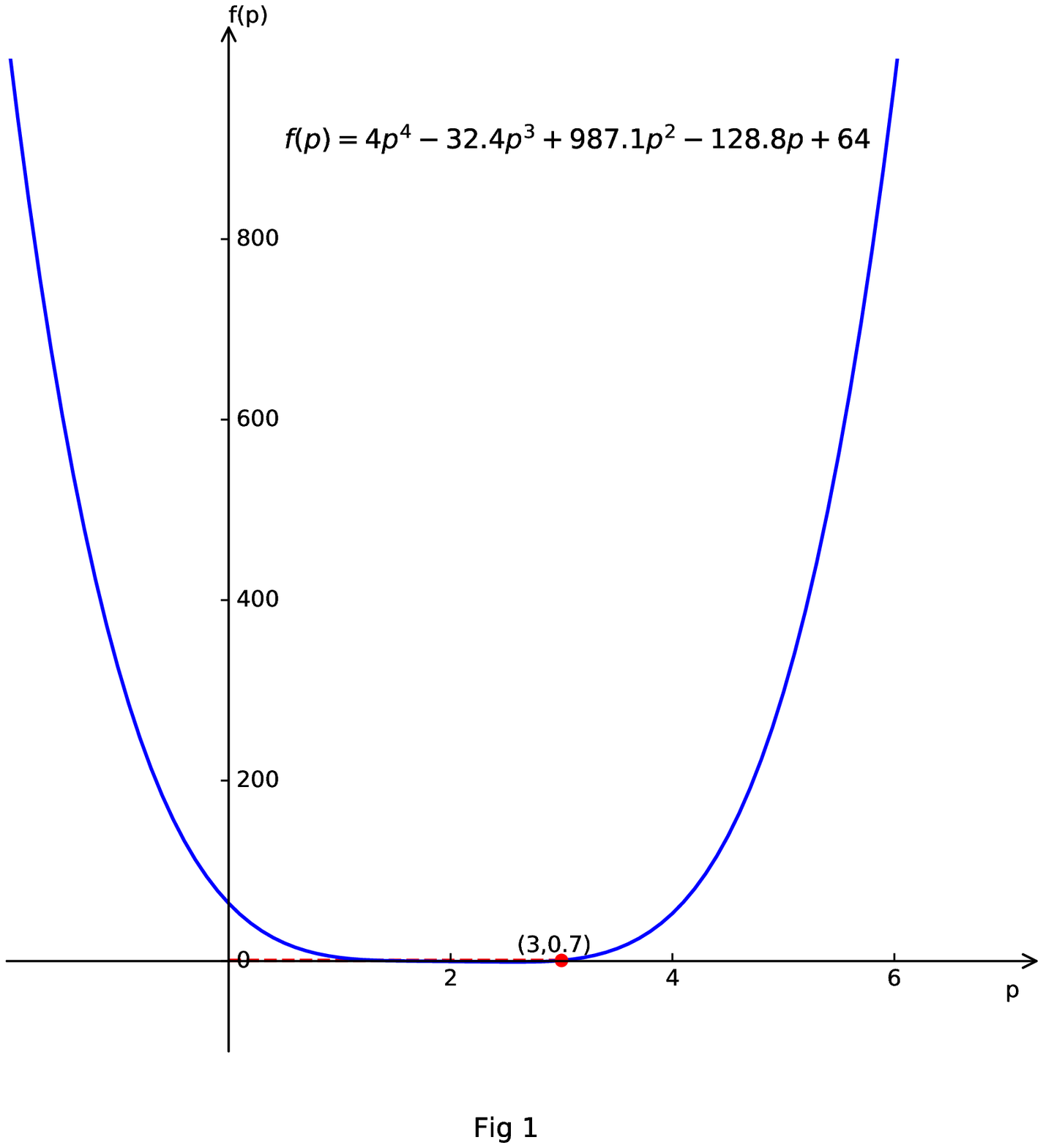}
	\quad\quad
	\includegraphics[height=8cm,width=8cm]{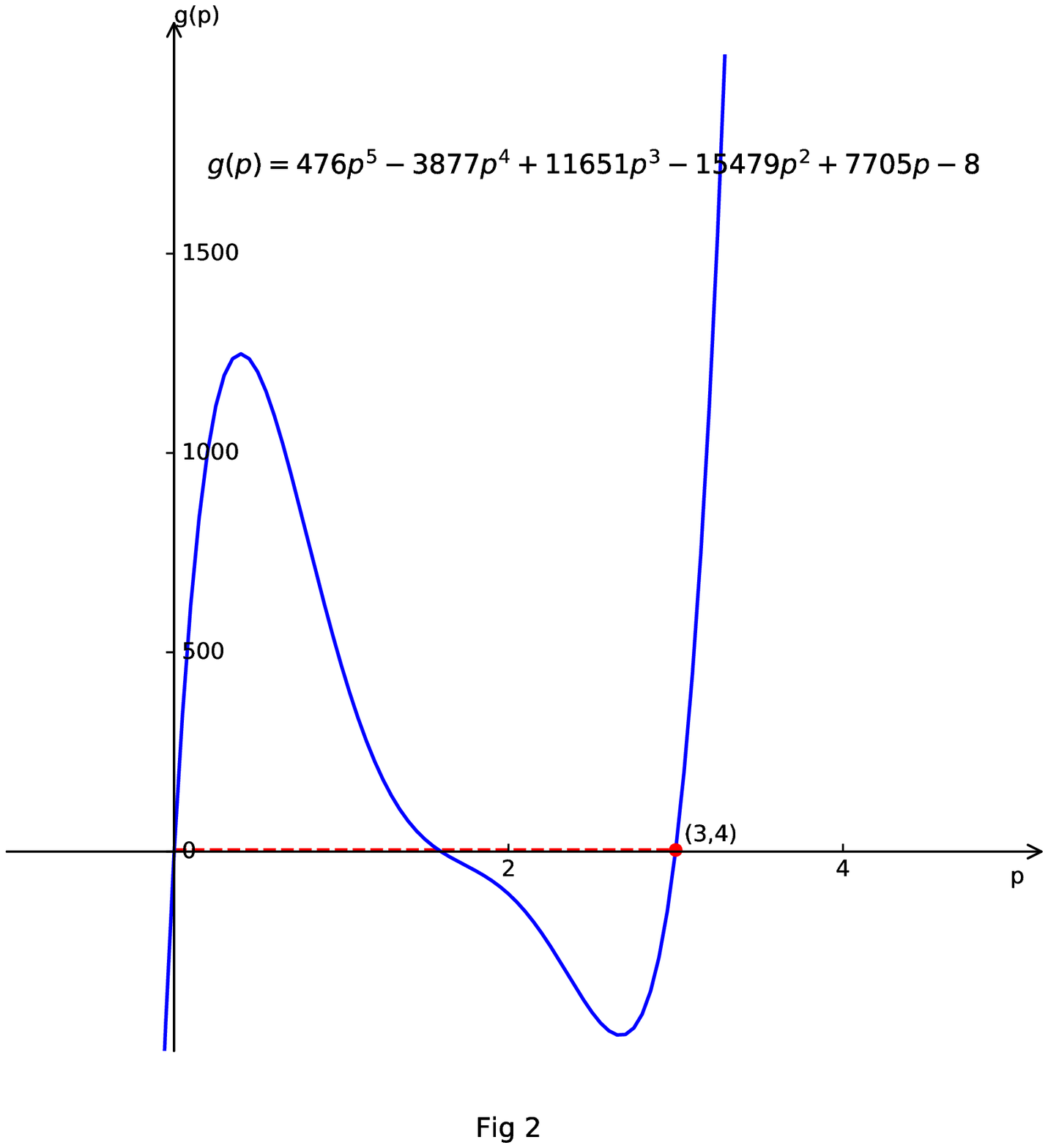}
\end{figure}

Since $f(p)\ge0,$ as $n\ge 120p^2$, we get that $nf(p)\ge 120p^2f(p).$
Thus we only need to show that
$120pf(p)-(p-1)^3(4p^2+p-8)\ge 0.$
Let $g(p)=120pf(p)-(p-1)^3(4p^2+p-8)=476p^5-3877p^4+11651p^3-15479p^2+7705p-8.$
Then, from the calculation by python (see Figure B), we can see that indeed, $g(p)\ge 0$ for all $p\ge 3$. This proves the desired inequality.
\QED

\section{On an immediate step in the proof of Lemma \ref{Lem:<p-joint-book}}
Here we want to show that the last inequality in the last paragraph of the proof of Lemma \ref{Lem:<p-joint-book} (see page 11) holds.
This is equivalent to show that
\begin{align*}
	h(n,p,r)\ge\frac{2(p-2)^2}{p(4p^2-11p+8)}n^2-\frac{(p-2)(2p-3)}{4p^2-11p+8}n-\frac{p(p-1)}{8(4p^2-11p+8)},
\end{align*}
where $h(n,p,r)=\frac{p(4p^2-11p+8)n^2}{8(p-1)^3}r^2-\frac{2(p-2)^2n^2+p(p-1)^2n}{4(p-1)^3}r+\frac{(p-2)^2}{2(p-1)^3}n^2-\frac{p-2}{2(p-1)}n.$
Reformulating the above inequality by using factorization, we let
\begin{align*}
	H(r)=&8p(p-1)^3(4p^2-11p+8)\cdot h(n,p,r)\\
	=&p^2(4p^2-11p+8)^2n^2\cdot r^2-\left(4p(p-2)^2(4p^2-11p+8)n^2+2p^2(p-1)^2(4p^2-11p+8)n\right)\cdot r\\
	&+4p(p-2)^2(4p^2-11p+8)n^2-4p(p-1)^2(p-2)(4p^2-11p+8)n.
\end{align*}
Then, it becomes to show that
$H(r)\ge 16(p-1)^3(p-2)^2n^2-8p(p-1)^3(p-2)(2p-3)n-p^2(p-1)^4.$
Since $H''(r)>0$, $H(r)$ is a convex quadratic function on $r$ and
minimized at $r=\frac{2(p-2)^2}{p(4p^2-11p+8)}+\frac{(p-1)^2}{(4p^2-11p+8)n}$ (i.e., the solution of the equation $H'(r)=0$).
Thus, we have
\begin{align*}
	H(r)&\ge H\left(\frac{2(p-2)^2}{p(4p^2-11p+8)}+\frac{(p-1)^2}{(4p^2-11p+8)n}\right)\\
	&=p^2(4p^2-11p+8)^2n^2\cdot\left(\frac{4(p-2)^4}{p^2(4p^2-11p+8)^2}+\frac{4(p-1)^2(p-2)^2}{p(4p^2-11p+8)^2n}+\frac{(p-1)^4}{(4p^2-11p+8)^2n^2}\right)\\
	&-\left(4p(p-2)^2(4p^2-11p+8)n^2+2p^2(p-1)^2(4p^2-11p+8)n\right)\cdot \left(\frac{2(p-2)^2}{p(4p^2-11p+8)}+\frac{(p-1)^2}{(4p^2-11p+8)n}\right)\\
	&+4p(p-2)^2(4p^2-11p+8)n^2-4p(p-1)^2(p-2)(4p^2-11p+8)n\\
	&=\left(4(p-2)^4-8(p-2)^4+4p(p-2)^2(4p^2-11p+8)\right)\cdot n^2+\big(4p(p-1)^2(p-2)^2-4p(p-1)^2(p-2)^2-\\
	&-4p(p-1)^2(p-2)^2-4p(p-1)^2(p-2)(4p^2-11p+8)\big)\cdot n+\left(p^2(p-1)^4-2p^2(p-1)^4\right)\\
	&=16(p-1)^3(p-2)^2n^2-8p(p-1)^3(p-2)(2p-3)n-p^2(p-1)^4,
\end{align*}
as we wanted. This verifies the inequality under consideration.
\QED

\bigskip

\bigskip

{\it E-mail address:} hjxhjl@mail.ustc.edu.cn

\medskip

{\it E-mail address:} fma@ustc.edu.cn

\medskip

{\it E-mail address:} jiema@ustc.edu.cn

\medskip

{\it E-mail address:} xinyang@stu.pku.edu.cn
	
\end{document}